\documentclass[fontsize=11pt,a4paper,reqno,numbers=noendperiod]{scrartcl}
\usepackage[greek,USenglish]{babel}
\usepackage[T1]{fontenc}
\usepackage[latin10]{inputenc}
\usepackage{lmodern}
\usepackage{amsmath,amssymb,amsthm}
\usepackage{dsfont}
\usepackage[obeyspaces]{url}
\usepackage{colortbl,color,xcolor}
\usepackage{graphicx,tikz,pgfplots}
\usepackage{centernot}
\usepackage{array}
\usepackage{enumitem}
\usepackage[babel]{microtype}
\usepackage{float}

\usepackage[bookmarks=true,plainpages=false,linktocpage,colorlinks=true,citecolor=green!80!black,linkcolor=red!70!black,filecolor=magenta,urlcolor=magenta,breaklinks,pdfauthor={Thomas Jahn, Margarita Spirova},pdftitle={On Bisectors in Normed Planes}]{hyperref}

\usetikzlibrary{calc,intersections,through,backgrounds,arrows,patterns,shapes.geometric,shapes.misc}

\newcommand{\RR}{\mathds{R}}
\newcommand{\ex}{\:\exists\:}
\newcommand{\mnorm}[1]{\left\lVert#1\right\rVert}
\newcommand{\setn}[1]{\left\{#1\right\}}
\newcommand{\setcond}[2]{\left\{#1 \:\middle\vert\: #2\right\}}
\newcommand{\setconds}[2]{\{#1 \:\vert\: #2\}}
\newcommand{\defeq}{\mathrel{\mathop:}=}
\newcommand{\dah}{i.e., }
\newcommand{\zB}{e.g., }

\newcommand{\strline}[2]{\left\langle#1,#2\right\rangle}
\newcommand{\clray}[2]{[#1,#2\rangle}

\theoremstyle{plain}
\newtheorem{Satz}{Theorem}[section]
\newtheorem{Kor}[Satz]{Corollary}

\newtheorem{Prop}[Satz]{Proposition}

\theoremstyle{definition}

\newtheorem{Bem}[Satz]{Remark}

\newtheorem*{msc}{2010 Mathematics Subject Classification}
\newtheorem*{keys}{Keywords}

\theoremstyle{remark}

\DeclareMathOperator{\co}{conv}
\DeclareMathOperator{\bd}{bd}
\DeclareMathOperator{\bisec}{bis}
\let \eps \varepsilon

\addtokomafont{caption}{\small}
\setkomafont{captionlabel}{\bfseries}

\begin{document}
\parindent 0pt
\title{On Bisectors in Normed Planes\footnote{accepted for publication in Contributions to Discrete Mathematics, published by the University of Calgary}}
\author{Thomas Jahn \and Margarita Spirova}
\date{}
\maketitle

\begin{abstract}
In this note, we completely describe the shape of the bisector of two given points in a two-dimensional normed vector space. More precisely, we show that, depending on the position of two given points with respect to the shape of the unit circle, the following holds: the bisector of a non-strict pair of points consists of two cones and a curve, which has properties similar to those of bisectors of strict pairs of points.
\end{abstract}

\begin{keys}
bisector, normed plane, norming functional
\end{keys}
\begin{msc}
46B20, 52A10
\end{msc}

\section{Introduction} The set of all points having equal distances from two given points is called their \emph{bisector} (or \emph{equidistance set}). The earliest contributions to this notion in arbitrary metric spaces were given by Mann \cite{Mann1935} and Busemann \cite{Busemann1941}, \cite{Busemann1944}. However, a deeper study of bisectors started only with the development of Computational Geometry, since geometric properties of bisectors are of decisive importance for the construction of Voronoi diagrams related to convex distance functions; see, \zB the surveys \cite{Aurenhammer1991}, \cite{AurenhammerKl2000}, and the papers \cite{ChewDr1985}, \cite{CorbalanMaReSa1993}, \cite{CorbalanMaRe1996}. The most contributions refer to strictly convex distance functions or to the so-called \emph{nice metrics} (for the definition of nice metrics see,  \zB \cite{Klein1989}). Recently, various results on bisectors were also obtained in Minkowski geometry (see \cite{Horvath2000}, \cite{Horvath2004}, \cite{HorvathMa2013}, \cite{HeMaWu2013}, and \cite{Vaeisaelae2013}). It is our aim to give a complete geometric description of bisectors based on distances which are induced by arbitrary norms.

Let $X=(\RR^2,\mnorm{\cdot})$ be a (\emph{normed} or) \emph{Minkowski plane}, \dah a two-dimensional normed vector space with unit ball $B$ which is a compact, convex set centered at its interior point $o$, the origin of $X$. If $X^\ast$ is the dual space of $X$, then a norm on $X^\ast$ is defined as $\mnorm{\phi}=\max_{\mnorm{x}=1}\phi(x)$. A \emph{norming functional of } $x$ is a $\phi\in X^\ast$ such that $\mnorm{\phi}=1$ and $\phi(x)=\mnorm{x}$, \dah the line $\phi^{-1}(1)=\setcond{y\in X}{\phi(y)=1}$ supports $B$ at $x$.

For two distinct points $x,y\in \RR^2$, we shall write
\begin{align*}
\strline{x}{y}&=\setcond{\lambda y+(1-\lambda)x}{\lambda\in \RR},\\
\clray{x}{y}&=\setcond{\lambda y+(1-\lambda)x}{\lambda\geq 0},\\
\clray{x}{y}_-&=\setcond{\lambda y+(1-\lambda)x}{\lambda\leq 0}
\end{align*}
for the straight line through $x$ and $y$, the ray starting at $x$ and passing through $y$, and its opposite ray, respectively.

The set of points with equal distances to two given points $p$ and $q$ is called their \emph{bisector} and denoted by $\bisec(p,q)$. The shape of the bisector depends on the shape of the unit ball $B$ of $X$ and the relative position of the two given points. Following \cite{Vaeisaelae2013}, we call the pair $(p,q)$ \emph{strict} if the unit circle does not contain any line segment parallel to $p-q$. Otherwise the pair $(p,q)$ is called \emph{non-strict}. In the strict case, the bisector $\bisec(p,q)$ of $p$ and $q$ shares the following properties (see, \zB \cite[\textsection~3.2]{MartiniSwWe2001} and \cite[\textsection~4.1]{MartiniSw2004}):

\begin{enumerate}[label={(\alph*)}, align=left,series=behaviour]
\item{Every straight line parallel to $p-q$ intersects $\bisec(p,q)$ exactly once.\label{first}}
\item{The set $\bisec(p,q)$ is homeomorphic to a straight line.\label{second}}
\item{For every point $x\in\bisec(p,q)$, the bisector is contained in the double cone with apex $x$ through $p$ and $q$.\label{last}}
\end{enumerate}

We will show that in the non-strict case the bisector is the union of two closed cones and a connected set $B_1$ that joins the apices of these cones. The properties of set $B_1$ are very similar to \ref{first}--\ref{last}:

\begin{enumerate}[label={(\alph*')}, align=left,series=behaviour-prime]
\item{Every straight line parallel to $p-q$ intersects $B_1$ in at most one point.\label{first-prime}}
\item{$B_1$ is homeomorphic to the unit interval $[0,1]$.\label{second-prime}}
\item{For every point $x\in B_1$, the set $B_1$ is contained in the double cone with apex $x$ through $p$ and $q$.\label{last-prime}}
\end{enumerate}
Scaling the unit disc by a factor $\eps$ yields a scaling of the Minkowski functional by $\eps^{-1}$. Hence the family of bisectors of pairs of distinct points is only determined by the shape of the unit ball and not by its size. This leads to the following construction of the bisector \cite[p. 18]{Ma2000}. Shrink the unit ball such that $(p+B)\cap (q+B)=\emptyset$. Choose $v_p\in p+\bd(B)$ and $v_q\in q+\bd(B)$ such that $\strline{v_p}{v_q}$ is a translate of $\strline{p}{q}$ and such that $\clray{p}{v_p}$ and $\clray{q}{v_q}$ have exactly one intersection point $z$; see Figure~\ref{fig:1}. The intercept theorem says
\begin{equation*}
\frac{\mnorm{z-p}}{\mnorm{v_p-p}}=\frac{\mnorm{z-q}}{\mnorm{v_q-q}},
\end{equation*}
\dah $z\in\bisec(p,q)$. Conversely, for $z\in \bisec(p,q)$ set $\setn{v_p}\defeq \clray{p}{z}\cap (p+\bd(B))$ and $\setn{v_q}\defeq \clray{q}{z}\cap (q+\bd(B))$. By the intercept theorem, $\strline{v_p}{v_q}$ is a translate of $\strline{p}{q}$. This construction works only for points in $\bisec(p,q)\setminus\strline{p}{q}$. Obviously, $\bisec(p,q)\cap \strline{p}{q}=\setn{\frac{1}{2}(p+q)}$.

\begin{figure}[h!]
\begin{center}
\begin{tikzpicture}[line cap=round,line join=round,>=triangle 45,x=1.2cm,y=1.2cm]
\draw[shift={(-1.5,0)}] (-1,0)--(0,1)--(1,1)--(1,0)--(0,-1)--(-1,-1)--cycle;
\draw[shift={(1.5,0)}] (-1,0)--(0,1)--(1,1)--(1,0)--(0,-1)--(-1,-1)--cycle;
\draw[dotted] (-2.5,0)--(2.5,0);
\draw[dotted] (-2.5,0.5)--(2.5,0.5);
\fill [color=black] (-1.5,0) circle(1.5pt) node[below]{$p$};
\fill [color=black] (-0.5,0.5) circle(1.5pt) node[above left]{$v_p$};
\fill [color=black] (1.5,0) circle(1.5pt) node[below]{$q$};
\fill [color=black] (1,0.5) circle(1.5pt) node[above]{$v_q$};
\draw[name path=ray1] (-1.5,0)--($(-1.5,0)!2.5!(-0.5,0.5)$);
\draw[name path=ray2] (1.5,0)--($(1.5,0)!2.5!(1,0.5)$);
\fill[color=black,name intersections={of=ray1 and ray2}](intersection-1)circle(1.5pt)node[above]{$z$};
\end{tikzpicture}
\end{center}
\caption{}\label{fig:1}
\end{figure}

\section{The main result and its proof}\label{chap:main}
Let $(p,q)$ be a non-strict pair in the normed plane $(\RR^2,\mnorm{\cdot})$. Without loss of generality, assume that $\strline{p}{q}$ is a horizontal line and $p$ lies left to $q$. Since $p+B$ and $q+B$ are translates of each other, there are two common supporting lines: $L_{\text{top}}$ above $\strline{p}{q}$ and $L_{\text{bottom}}$ below $\strline{p}{q}$; see Figure~\ref{fig:2}. Let
\begin{align}
L_{\text{top}}\cap(\bd(B)+p)&=[t_p,t_p^\prime],\notag\\
L_{\text{top}}\cap(\bd(B)+q)&=[t_q^\prime,t_q],\notag\\
L_{\text{bottom}}\cap(\bd(B)+p)&=[b_p,b_p^\prime],\notag\\
L_{\text{bottom}}\cap(\bd(B)+q)&=[b_q^\prime,b_q],\label{eq:top_bottom}
\end{align}
such that the alignment of the points on $L_{\text{top}}$ and $L_{\text{bottom}}$ is as depicted in Figure~\ref{fig:2}.

\begin{figure}[h!]
\begin{center}
\begin{tikzpicture}[line cap=round,line join=round,>=triangle 45,x=1.0cm,y=1.0cm]
\draw[shift={(-1.5,0)}] (-1,0)--(0,1)--(1,1)--(1,0)--(0,-1)--(-1,-1)--cycle;
\draw[shift={(1.5,0)}] (-1,0)--(0,1)--(1,1)--(1,0)--(0,-1)--(-1,-1)--cycle;
\draw (-2.5,1)--(2.5,1);
\draw[dotted] (-2.5,0)--(2.5,0);
\draw (-2.5,-1)--(2.5,-1);
\fill [color=black] (-1.5,0) circle(1.5pt) node[below]{$p$};
\fill [color=black] (1.5,0) circle(1.5pt) node[below]{$q$};
\fill [color=black] (-1.5,1) circle(1.5pt) node[above]{$t_p$};
\fill [color=black] (-0.5,1) circle(1.5pt) node[above]{$t_p^\prime$};
\fill [color=black] (1.5,1) circle(1.5pt) node[above]{$t_q^\prime$};
\fill [color=black] (2.5,1) circle(1.5pt) node[above]{$t_q$};
\fill [color=black] (-1.5,-1) circle(1.5pt) node[below]{$b_p^\prime$};
\fill [color=black] (-2.5,-1) circle(1.5pt) node[below]{$b_p$};
\fill [color=black] (1.5,-1) circle(1.5pt) node[below]{$b_q$};
\fill [color=black] (0.5,-1) circle(1.5pt) node[below]{$b_q^\prime$};
\draw (2.6,1) node[right]{$L_{\text{top}}$};
\draw (2.6,-1) node[right]{$L_{\text{bottom}}$};
\end{tikzpicture}
\end{center}
\caption{}\label{fig:2}
\end{figure}
In the following considerations, two points will play an important role: the intersection point $s_t$ of the rays $\clray{p}{t_p^\prime}$ and $\clray{q}{t_q^\prime}$, and the intersection point $s_b$ of the rays $\clray{p}{b_p^\prime}$ and $\clray{q}{b_q^\prime}$.

The first part of our theorem is a known result (see, \zB \cite[Proposition~22]{MartiniSw2004}), but we will give a short proof of it in the context of the remaining part of the theorem. We also refer to \cite[Lemma~2.1.1.1]{Ma2000} for this result, but concerning strict pair of points.

\begin{Satz}\label{thm:parallel}
Let $(p,q)$ be a non-strict pair in the normed plane $(\RR^2,\mnorm{\cdot})$. Then the bisector $\bisec(p,q)$ has interior points and is fully contained in the interior of the bent strip $\co(\clray{p}{t_p}\cup\clray{q}{t_q})\cup \co(\clray{p}{b_p}\cup\clray{q}{b_q})$. More precisely, we have $\bisec(p,q)=B_1\cup B_2\cup B_3$, where
\begin{itemize}
\item{there is a homeomorphism $f:[0,1]\to B_1$ with $f(0)=s_b$ and $f(1)=s_t$,}
\item{$B_2=\co(\clray{s_t}{p}_-\cup\clray{s_t}{q}_-)$,}
\item{$B_3=\co(\clray{s_b}{p}_-\cup\clray{s_b}{q}_-)$.}
\end{itemize}
\end{Satz}

\begin{proof}
The bisector $\bisec(p,q)$ is symmetric with respect to the midpoint $\frac{1}{2}(p+q)$. Hence it suffices to have a look at that part $B_t$ of $\bisec(p,q)$ which lies in the closed half-plane above $\strline{p}{q}$. First we show that $B_t$ is contained in the mentioned strip. Let $z\in B_t$; see Figure~\ref{fig:3}. If $z\in\strline{p}{q}$, \dah $z=\frac{1}{2}(p+q)$, we are done. Otherwise, we set $\setn{v_p}\defeq\clray{p}{z}\cap (\bd(B)+p)$, $\setn{v_q}\defeq\clray{q}{z}\cap (\bd(B)+q)$.
\begin{figure}[h!]
\begin{center}
\begin{tikzpicture}[line cap=round,line join=round,>=triangle 45,x=1.2cm,y=1.2cm]
\draw[shift={(-1.5,0)}] (-1,0)--(0,1)--(1,1)--(1,0)--(0,-1)--(-1,-1)--cycle;
\draw[shift={(1.5,0)}] (-1,0)--(0,1)--(1,1)--(1,0)--(0,-1)--(-1,-1)--cycle;
\draw (-2.5,0)--(2.5,0);
\draw[dotted] (-2.5,0.5)--(2.5,0.5);
\fill [color=black] (-1.5,0) circle(1.5pt) node[below]{$p$};
\fill [color=black] (-2,0.5) circle(1.5pt) node[above]{$\tilde{v}_p$};
\fill [color=black] (-0.5,0.5) circle(1.5pt) node[above left]{$v_p$};
\fill [color=black] (1.5,0) circle(1.5pt) node[below]{$q$};
\fill [color=black] (2.5,0.5) circle(1.5pt) node[above left]{$\tilde{v}_q$};
\fill [color=black] (1,0.5) circle(1.5pt) node[above]{$v_q$};
\draw[name path=ray1] (-1.5,0)--($(-1.5,0)!2.5!(-0.5,0.5)$);
\draw[name path=ray2] (1.5,0)--($(1.5,0)!2.5!(1,0.5)$);
\fill[color=black,name intersections={of=ray1 and ray2}](intersection-1)circle(1.5pt)node[above]{$z$};
\end{tikzpicture}
\end{center}
\caption{}\label{fig:3}
\end{figure}

By the intercept theorem, $\strline{v_p}{v_q}$ is a translate of $\strline{p}{q}$, and the rays $\clray{p}{v_p}$, $\clray{q}{v_q}$ intersect (namely in $z$). This would not be possible if $z$ were above $\strline{p}{q}$ but outside the interior of $\co(\clray{p}{t_p}\cup\clray{q}{t_q})$. Another consequence of the intercept theorem is $s_t\in B_t$. We will show that $B_t=B_t^1\cup B_2$, where $B_t^1$ is a curve with endpoints $\frac{1}{2}(p+q)$ and $s_b$ which is homeomorphic to the closed unit interval $[0,1]$. Let $R_1\defeq \co(\clray{p}{t_p}\cup\clray{q}{t_q})\setminus\strline{p}{q}$ be the upper part of the relevant strip; see Figure~\ref{fig:4}.

\begin{figure}[H]
\begin{center}
\begin{tikzpicture}[line cap=round,line join=round,>=triangle 45,x=1.0cm,y=1.0cm]
\fill[color=black,fill=black,fill opacity=0.1] (-1.5,4) -- (-1.5,0) -- (1.5,0) -- (5.5,4) -- cycle;
\draw[shift={(-1.5,0)}] (-1,0)--(0,1)--(1,1)--(1,0)--(0,-1)--(-1,-1)--cycle;
\draw[shift={(1.5,0)}] (-1,0)--(0,1)--(1,1)--(1,0)--(0,-1)--(-1,-1)--cycle;
\fill [color=black] (-1.5,0) circle(1.5pt) node[left]{$p$};
\fill [color=black] (1.5,0) circle(1.5pt) node[right]{$q$};
\fill [color=black] (-1.5,1) circle(1.5pt) node[above left]{$t_p$};
\fill [color=black] (-0.5,1) circle(1.5pt) node[above]{$t_p^\prime$};
\fill [color=black] (1.5,1) circle(1.5pt) node[above left]{$t_q^\prime$};
\fill [color=black] (2.5,1) circle(1.5pt) node[above]{$t_q$};
\fill [color=black] (-1.5,-1) circle(1.5pt) node[below right]{$b_p^\prime$};
\fill [color=black] (-2.5,-1) circle(1.5pt) node[below]{$b_p$};
\fill [color=black] (1.5,-1) circle(1.5pt) node[below right]{$b_q$};
\fill [color=black] (0.5,-1) circle(1.5pt) node[below]{$b_q^\prime$};
\draw[color=black] (3,3) node{$R_1$};
\draw (-1.5,0)--(1.5,0);
\draw[name path=strline1] ($(-0.5,1)!1.2!(-2.5,-1)$)--($(-2.5,-1)!2.5!(-0.5,1)$);
\draw ($(-1.5,1)!1.2!(-1.5,-1)$)--($(-1.5,-1)!2.5!(-1.5,1)$);
\draw ($(2.5,1)!1.2!(0.5,-1)$)--($(0.5,-1)!2.5!(2.5,1)$);
\draw[name path=strline2] ($(1.5,1)!1.2!(1.5,-1)$)--($(1.5,-1)!2.5!(1.5,1)$);
\fill[color=black,name intersections={of=strline1 and strline2}](intersection-1)circle(1.5pt)node[right]{$s_t$};
\end{tikzpicture}
\end{center}
\caption{}\label{fig:4}
\end{figure}
Let
\begin{align*}
B_t^1&=\bigcup_{\lambda\in (\frac{1}{2}\mnorm{p-q},\mnorm{s_t-p}]}\setcond{z\in R_1}{\mnorm{z-p}=\mnorm{z-q}=\lambda},\\
B_t^2&=\bigcup_{\lambda\in [\mnorm{s_t-p},\infty)}\setcond{z\in R_1}{\mnorm{z-p}=\mnorm{z-q}=\lambda}.
\end{align*}
Obviously, $B_t=B_t^1\cup B_t^2$ and $B_t^1\cap B_t^2=\setn{s_t}$. Like in the proof of \cite[Lemma~2.1.1.1]{Ma2000}, one can show that there is a homeomorphism from $B_t^1$ to the closed arc of $\bd(B)+p$ between $t_p^\prime$ and $p+\frac{q-p}{\mnorm{q-p}}$ that does not contain $t_p$. Let $G_t$ be the translate of $\strline{p}{q}$ through $s_t$. Clearly, $B_t^2$ lies in the closed half-plane $R_2$ above $G_t$, \dah $B_t^2=B_t\cap R_2$. We will show that $B_2=B_t^2$.
\begin{figure}[h!]
\begin{center}
\begin{tikzpicture}[line cap=round,line join=round,>=triangle 45,x=1.0cm,y=1.0cm]
\clip (-2.6,-0.3) rectangle(6.6,4.9);
\fill[color=black,fill=black,fill opacity=0.1] (-1.5,5) -- (-1.5,0)--(3.5,5) -- cycle;
\fill[color=black,fill=black,fill opacity=0.1] (3.5,5) -- (1.5,3) -- (1.5,5) -- cycle;
\fill[color=black,fill=black,fill opacity=0.1](1.5,5) --(1.5,0)--(6.5,5)-- cycle;
\draw[shift={(-1.5,0)}] (-1,0)--(0,1)--(1,1)--(1,0)--(0,-1)--(-1,-1)--cycle;
\draw[shift={(1.5,0)}] (-1,0)--(0,1)--(1,1)--(1,0)--(0,-1)--(-1,-1)--cycle;
\draw[color=black] (4,4) node{$C(q,\phi)$};
\draw[color=black] (0,4) node{$C(p,\phi)$};
\draw[color=black] (2,4) node{$B_2$};
\fill [color=black] (-1.5,0) circle(1.5pt) node[left]{$p$};
\fill [color=black] (1.5,0) circle(1.5pt) node[right]{$q$};
\fill [color=black] (-1.5,1) circle(1.5pt) node[above left]{$t_p$};
\fill [color=black] (-0.5,1) circle(1.5pt) node[above]{$t_p^\prime$};
\fill [color=black] (1.5,1) circle(1.5pt) node[above left]{$t_q^\prime$};
\fill [color=black] (2.5,1) circle(1.5pt) node[above]{$t_q$};
\draw (-1.5,0)--(1.5,0);
\draw (-2,3)--(6,3);
\draw[name path=strline1] ($(-0.5,1)!1.2!(-2.5,-1)$)--($(-2.5,-1)!3!(-0.5,1)$);
\draw ($(-1.5,1)!1.2!(-1.5,-1)$)--($(-1.5,-1)!3!(-1.5,1)$);
\draw ($(2.5,1)!1.2!(0.5,-1)$)--($(0.5,-1)!3!(2.5,1)$);
\draw[name path=strline2] ($(1.5,1)!1.2!(1.5,-1)$)--($(1.5,-1)!3!(1.5,1)$);
\fill[color=black,name intersections={of=strline1 and strline2}](intersection-1)circle(1.5pt)node[below right]{$s_t$};
\draw (6,3) node[below]{$G_t$};
\end{tikzpicture}
\end{center}
\caption{}\label{fig:5}
\end{figure}
The inclusion $B_2\subset R_2$ is evident. Next, we show $B_2\subset\bisec(p,q)$. Due to the Hahn--Banach theorem there is a unique norming functional $\phi$ of $\frac{1}{2}(t_p+t_p^\prime)-p$. The level sets of the functional $\phi$ are straight lines parallel to $p-q$, and $\phi$ takes the value $1$ at the points of $L_{\text{top}}-p$. Let us define the cone $C(x,\phi)=x+\setconds{a\in \RR^2}{\phi(a)=\mnorm{a}}$, \dah $C(x,\phi)$ is the translate by $x$ of the union of rays from the origin through the exposed face $\phi^{-1}(1)\cap B$ of the unit ball $B$ of $(\RR^2,\mnorm{\cdot})$, see \cite[Definition~3.4]{MartiniSwWe2002} and Figure~\ref{fig:5}. Then it is easy to see that $B_2=C(p,\phi)\cap C(q,\phi)$. We have
\begin{align*}
z\in C(p,\phi)\cap C(q,\phi)&\Longleftrightarrow \ex a,b\in X: \begin{cases}\phi(a)=\mnorm{a},\phi(b)=\mnorm{b},\\z-p=a, z-q=b\end{cases}\\
&\Longrightarrow \begin{cases}\phi(z-p)=\phi(a)=\mnorm{a}=\mnorm{z-p},\\\phi(z-q)=\phi(b)=\mnorm{b}=\mnorm{z-q}\end{cases}\\
&\stackrel{\star}{\Longrightarrow} \mnorm{z-p}= \phi(z)-\phi(p)=\phi(z)-\phi(q)=\mnorm{z-q}\\
&\Longrightarrow z\in\bisec(p,q).
\end{align*}
This shows the inclusion $B_2\subset R_2\cap \bisec(p,q)=B_t^2$. In step $\stackrel{\star}{\Longrightarrow}$, we used $\phi(p)-\phi(q)=\phi(p-q)=0$ which follows from the choice of $\phi$. Conversely, if $z\in B_t^2$, then $z\in C(p,\phi)\cup C(q,\phi)$, that is, $\phi$ is a norming functional of $z-p$ or of $z-q$. The equality $\phi(z-p)=\phi(z-q)$ is true by choice of $\phi$. If $\phi$ is a norming functional of $z-p$, say, but not of $z-q$, then $\mnorm{z-p}=\phi(z-p)=\phi(z-q)\neq\mnorm{z-q}$, which contradicts $z\in\bisec(p,q)$.
\end{proof}
Note that the second part of this proof can be simplified using the fact that the bisector is convex in the direction of $p-q$ (see \cite[Lemmas 1 and 2]{Horvath2000}).

Using the notation from Theorem~\ref{thm:parallel} and its proof we rewrite this theorem in the following form.

\begin{Kor}
Let $(p,q)$ be a non-strict pair in the normed plane $(\RR^2,\mnorm{\cdot})$, and $\phi$ be a (unique) norming functional of $\frac{1}{2}(t_p+t_p^\prime)-p$. Then $\bisec(p,q)=B_1\cup B_2\cup B_3$, where
\begin{itemize}
\item{there is a homeomorphism $f:[0,1]\to B_1$ with $f(0)=s_b$ and $f(1)=s_t$,}
\item{$B_2=C(p,\phi)\cap C(q,\phi)$,}
\item{$B_3=C(p, -\phi)\cap C(q, -\phi)$.}
\end{itemize}
\end{Kor}

Now we show in two propositions that $B_t$ obeys \ref{first-prime} and \ref{last-prime}. Again, the central symmetry of the bisector allows a restriction to the upper part $B_t$.

The first proposition is a local version of \cite[Proposition~15]{MartiniSwWe2001}.

\begin{Prop}\label{pro1}
Every translate $G$ of $\strline{p}{q}$ above $\strline{p}{q}$ and not above $G_t$ intersects $\bisec(p,q)$ in exactly one point.
\end{Prop}

\begin{proof}
Assume $G\cap \bisec(p,q)$ contains two distinct points $z_1$, $z_2$ such that $[p,z_1]\cap[z_2,q]=\emptyset$; see Figure~\ref{fig:6}. Now \cite[Proposition~7]{MartiniSwWe2001} yields
\begin{equation}
\mnorm{z_2-p}+\mnorm{z_1-q}\geq \mnorm{z_1-p}+\mnorm{z_2-q}.\label{eq:diagonals_vs_sides2}
\end{equation}
By assumption, equation \eqref{eq:diagonals_vs_sides2} holds with equality, but $[\frac{z_2-p}{\mnorm{z_2-p}},\frac{z_1-q}{\mnorm{z_1-q}}]\centernot\subset\bd(B)$. Hence, the assumption that $G\cap \bisec(p,q)$ contains at least two points is wrong.
\end{proof}
\begin{figure}[h!]
\begin{center}
\begin{tikzpicture}[line cap=round,line join=round,>=triangle 45,x=1.0cm,y=1.0cm]
\clip (-2.6,-0.3) rectangle(6.6,4);
\draw[name path=hex1,shift={(-1.5,0)}] (-1,0)--(0,1)--(1,1)--(1,0)--(0,-1)--(-1,-1)--cycle;
\draw[name path=hex2,shift={(1.5,0)}] (-1,0)--(0,1)--(1,1)--(1,0)--(0,-1)--(-1,-1)--cycle;
\fill [color=black] (-1.5,0) circle(1.5pt) node[left]{$p$};
\fill [color=black] (1.5,0) circle(1.5pt) node[right]{$q$};
\fill [color=black] (1.4,1.8) circle(1.5pt) node[above right]{$z_2$};
\fill [color=black] (0.9,1.8) circle(1.5pt) node[above]{$z_1$};
\draw (-2,1.8)--(6,1.8) node[below]{$G$};
\draw (-2,3)--(6,3) node[below]{$G_t$};
\draw [name path=strline0] (-1.5,0)--(1.4,1.8);
\draw [name path=strline3] (1.5,0)--(0.9,1.8);
\draw[dotted] (-1.5,0)--(1.5,0);
\draw[name path=strline1] ($(-0.5,1)!1.2!(-2.5,-1)$)--($(-2.5,-1)!3!(-0.5,1)$);
\draw ($(-1.5,1)!1.2!(-1.5,-1)$)--($(-1.5,-1)!3!(-1.5,1)$);
\draw ($(2.5,1)!1.2!(0.5,-1)$)--($(0.5,-1)!3!(2.5,1)$);
\draw[name path=strline2] ($(1.5,1)!1.2!(1.5,-1)$)--($(1.5,-1)!3!(1.5,1)$);
\fill[color=black,name intersections={of=strline1 and strline2}](intersection-1)circle(1.5pt)node[above left]{$s_t$};
\fill[color=black,name intersections={of=hex1 and strline0,name=schnitt}](schnitt-1)circle(1.5pt);
\fill[color=black,name intersections={of=hex2 and strline3,name=cut}](cut-1)circle(1.5pt);
\end{tikzpicture}
\end{center}
\caption{}\label{fig:6}
\end{figure}
The  next proposition is a local version on \cite[Proposition~17]{MartiniSwWe2001} (see also \cite[Theorem~1]{Holub1975}).
\begin{Prop}\label{pro2}
For all $z\in B_t^1$, the bisector part $B_t^1$ is contained in the double cone
\begin{equation*}
\setcond{z+\lambda (p-z)+\mu(q-z)}{\lambda, \mu\in\RR,\lambda\mu\geq 0}.
\end{equation*}
In particular, $B_t^1\subset \co(\setn{p,q,s_t})$.
\end{Prop}
\begin{proof}
Assume the converse statement. Then there is a point $w=z+\lambda (p-z)+\mu(q-z)\in B_t^1$ with $\lambda\mu<0$, say $\lambda <0$ and $\mu>0$; see Figure~\ref{fig:7}. From \cite[Proposition~7]{MartiniSwWe2001} it follows that
\begin{equation}
\mnorm{w-p}+\mnorm{z-q}\geq \mnorm{w-q}+\mnorm{z-p}, \label{eq:diagonals_vs_sides1}
\end{equation}
with equality if and only if $[\frac{w-p}{\mnorm{w-p}},\frac{z-q}{\mnorm{z-q}}]\subset\bd(B)$.
\begin{figure}[h!]
\begin{center}
\begin{tikzpicture}[line cap=round,line join=round,>=triangle 45,x=1.0cm,y=1.0cm]
\clip (-2.6,-0.3) rectangle(6.6,4);
\draw[name path=hex,shift={(-1.5,0)}] (-1,0)--(0,1)--(1,1)--(1,0)--(0,-1)--(-1,-1)--cycle;
\draw[name path=hexhex,shift={(1.5,0)}] (-1,0)--(0,1)--(1,1)--(1,0)--(0,-1)--(-1,-1)--cycle;
\fill [color=black] (-1.5,0) circle(1.5pt) node[left]{$p$};
\fill [color=black] (1.5,0) circle(1.5pt) node[right]{$q$};
\fill [color=black] (2.7,2.0) circle(1.5pt) node[above]{$w$};
\fill [color=black] (-0.7,2.4) circle(1.5pt) node[left]{$z$};
\draw [name path=strline0] (-1.5,0)--(2.7,2.0);
\draw[dotted] (-1.5,0)--(1.5,0);
\draw (-2,3)--(6,3)node[below]{$G_t$};
\draw[name path=strline3] ($(1.5,0)!1.7!(-0.7,2.4)$)--($(-0.7,2.4)!1.2!(1.5,0)$);
\draw ($(-1.5,0)!1.7!(-0.7,2.4)$)--($(-0.7,2.4)!1.2!(-1.5,0)$);
\draw[name path=strline1] ($(-0.5,1)!1.2!(-2.5,-1)$)--($(-2.5,-1)!3!(-0.5,1)$);
\draw[name path=strline2] ($(1.5,1)!1.2!(1.5,-1)$)--($(1.5,-1)!3!(1.5,1)$);
\draw ($(-1.5,1)!1.2!(-1.5,-1)$)--($(-1.5,-1)!3!(-1.5,1)$);
\draw ($(2.5,1)!1.2!(0.5,-1)$)--($(0.5,-1)!3!(2.5,1)$);
\fill[color=black,name intersections={of=strline1 and strline2}](intersection-1)circle(1.5pt)node[above left]{$s_t$};
\fill[color=black,name intersections={of=hex and strline0,name=schnitt}](schnitt-1)circle(1.5pt);
\fill[color=black,name intersections={of=hexhex and strline3,name=cut}](cut-1)circle(1.5pt);
\end{tikzpicture}
\end{center}
\caption{}\label{fig:7}
\end{figure}

Since $z,w\in\bisec(p,q)$, equation \eqref{eq:diagonals_vs_sides1} holds with equality, but $[\frac{w-p}{\mnorm{w-p}},\frac{z-q}{\mnorm{z-q}}]\centernot\subset\bd(B)$. Hence the assumption $w\in\bisec(p,q)$ is wrong.
\end{proof}

The next statement is immediately consequence of Theorem~\ref{thm:parallel}, Proposition~\ref{pro1}, and Proposition~\ref{pro2}.

\begin{Kor}
Let $\bisec(p, q)$ be the bisector of two different points $p$ and $q$. Then there exists a simple curve $B(p,q)\subseteq \bisec(p,q)$ through the midpoint of $p$ and $q$ and symmetric with respect to this midpoint which is homeomorphic to $\RR$ such that for every $x\in B(p,q)$ the curve $B(p,q)$ belongs to the double cone with apex $x$ and through $p$ and $q$. This curve separates the plane into two connected parts $B(p,q)^+$ and $B(p, q)^{-}$ such that whenever $y\in B(p, q)^+$, then $\mnorm{p-y}\leq\mnorm{q-y}$, and whenever $y\in B(p,q)^{-}$, then $\mnorm{p-y}\geq\mnorm{q-y}$. Moreover, if $x\in \bisec(p,q)$, then the curve $B(p,q)$ can be constructed to pass through $x$.
\end{Kor}

For the strictly convex case in spaces of higher dimensions these properties of the bisector can be found in \cite[Lemma~1, Theorem~2]{Horvath2000}.

\begin{Bem}
Recall the properties \ref{first}--\ref{last} of a bisector of a strict pair, as mentioned in the Introduction. V{\"a}is{\"a}l{\"a} \cite[Theorem~2.4]{Vaeisaelae2013} uses these properties for representing the bisector as the graph of a Lipschitz-continuous function $\RR\to\RR$. A slight modification of the corresponding proof can be used to show that, in the non-strict case, $B_1$ can be viewed as the graph of a Lipschitz-continuous function $[0,1]\to \RR$. We give a brief explanation thereof. First of all, we have a look at the strict case and denote the uppermost points of $p+B$ and $q+B$ by $t_p$ and $t_q$, respectively. The fundament of this V{\"a}is{\"a}l{\"a}'s theorem is \cite[Lemma~2.3]{Vaeisaelae2013}, where the following statement can be found (adapted to our language). \emph{If $(p,q)$ is a strict pair, $z\in\bisec(p,q)$ lies on or above the straight line $\strline{p}{q}$, and $z^\prime$ lies in the component of $\bisec(p,q)\setminus\setn{z}$ above $z$, then $z^\prime$ lies above $\clray{z}{t_p}\cup\clray{z}{t_q}$.}

Now the transition to the non-strict case is easier. If we apply the denotation introduced at the beginning of this section, the following statement for the non-strict case can be proved similarly to \cite[Lemma~2.3]{Vaeisaelae2013}. \emph{If $(p,q)$ is a non-strict pair, $z\in B_1$, and $z^\prime$ lies in the component of $B_1\setminus\setn{z}$ above $z$, then $z^\prime$ lies above $\clray{z}{t_p}\cup\clray{z}{t_q}$.}
\end{Bem}

\providecommand{\bysame}{\leavevmode\hbox to3em{\hrulefill}\thinspace}
\vspace{0.3cm}
\begin{tabular}{l}
Thomas Jahn\\
Fakult\"at f\"ur Mathematik, TU Chemnitz\\
D-09107 Chemnitz, GERMANY\vspace{0.1cm}\\
E-mail: thomas.jahn\raisebox{-1.5pt}{@}mathematik.tu-chemnitz.de
\end{tabular}\vspace{0.3cm}

\begin{tabular}{l}
Margarita Spirova\\
Fakult\"at f\"ur Mathematik, TU Chemnitz\\
D-09107 Chemnitz, GERMANY\vspace{0.1cm}\\
E-mail: margarita.spirova\raisebox{-1.5pt}{@}mathematik.tu-chemnitz.de
\end{tabular}
\end{document}